\documentclass[12pt]{article}

\usepackage{geometry}
\usepackage[utf8]{inputenc}
\usepackage[T1]{fontenc}
\usepackage{hyperref}
\usepackage{amsmath}
\usepackage{amssymb}
\usepackage{amscd}
\usepackage{amsthm}
\usepackage{amsfonts}
\usepackage{amscd}
\usepackage{enumerate}
\usepackage{graphicx}
\usepackage[all,cmtip]{xy}
\usepackage{cite}

\DeclareMathOperator\Stab {Stab}
\renewcommand\ker {\mathop{\mathrm{Ker}}}

\DeclareMathOperator\Frob {Frob}

\DeclareMathOperator\Gal {Gal}
\DeclareMathOperator\Sp {Spec\,}
\DeclareMathOperator\End {End}

\newcommand{\tens}{\otimes}

\newcommand{\al}{\alpha}
\newcommand{\ds}{\displaystyle}
\newcommand{\e}{\varepsilon}

\newcommand{\inj}{\hookrightarrow}

\newcommand{\ka}{\kappa}
\newcommand{\ov}{\overline}

\newcommand{\s}{\sigma}

\newcommand{\wh}{\widehat}
\newcommand{\wt}{\widetilde}

\newcommand{\cart}{\ar@{}[rd]|-\square}

\newcommand{\N}{\mathbb{N}}
\newcommand{\Z}{\mathbb{Z}}
\newcommand{\Q}{\mathbb{Q}}

\newcommand{\C}{\mathbb{C}}

\newcommand{\sh}{{\mathrm{sh}}}
\newcommand{\nr}{{\mathrm{nr}}}
\newcommand{\ta}{{\mathrm{tame}}}

\newcommand{\tor}{{\mathrm{tors}}}
\newcommand{\pp}{{p'-\mathrm{tors}}}
\newcommand{\ptor}{{p-\mathrm{tors}}}

\newcommand{\mb}[1]{{\mathbb{#1}}}
\newcommand{\mc}[1]{{\mathcal{#1}}}
\newcommand{\mf}[1]{{\mathfrak{#1}}}

\newcommand{\ssi}{\Leftrightarrow}

\author{Cyrille Corpet\footnotemark[2]}
\title{Torsion subschemes and the Tate-Voloch conjecture over $p$-adic fields}

\begin{document}

\maketitle
\footnotetext[2]{Doctorant contractuel, Institut de Math\'ematiques, \'Equipe \'Emile Picard, Universit\'e Paul Sabatier, 118 Route de Narbonne, 31062 Toulouse Cedex 9, France

E-mail:~cyrille.corpet@math.univ-toulouse.fr}
\theoremstyle{plain}
\newtheorem{thm}{Theorem}[section]
\newtheorem{lem}[thm]{Lemma}
\newtheorem{prop}[thm]{Proposition}
\newtheorem{coro}[thm]{Corollary}
\newtheorem{conj}[thm]{Conjecture}
\newtheorem*{claim}{Claim}

\theoremstyle{definition}
\newtheorem{defi}{Definition}[section]
\newtheorem{exmp}{Example}[section]

\theoremstyle{remark}
\newtheorem*{rems}{Remarks}
\newtheorem{rem}{Remark}
\newtheorem*{note}{Note}
\newtheorem{case}{Case}
\begin{quotation}
\subsubsection*{Abstract}
{
\hspace \parindent We give a new proof of the Tate-Voloch conjecture, in the situation where the ambient variety
is a semiabelian variety defined over $\ov \Q_p$. Our proof is new in the sense that it
avoids any reference to algebraic model theory or $p$-adic Hodge theory.
}
\end{quotation}
\section*{Introduction}
\addcontentsline{toc}{section}{Introduction}

In this article, we use some of the techniques of the papers \cite{PinkRo,MM2ML} to prove a weak form of the Tate-Voloch conjecture.

\bigskip

We will freely use the following notations: $p$ is a prime number, fixed once and for all; $\C_p$ is the completion of $\ov {\Q_p}$, the algebraic closure of the field of $p$-adic numbers $\Q_p$. Its normalized $p$-adic norm is written as $\lvert\cdot\rvert_p$.

If $G$ is a group, then $G_\tor$ is the torsion subgroup of $G$ and $G_\ptor$ (resp. $G_\pp$) the  torsion subgroup whose elements have torsion of order a power of $p$ (resp. prime to $p$).

\bigskip

The Tate-Voloch conjecture is a statement about the ``$p$-adic distance'' between a closed subset and the torsion points in a semiabelian variety over $\C_p$. There is no canonical definition of such a distance, but for any of the natural definitions that will be described in the first section, the following statement makes sense (and its truth is actually independent of the distance used):
\begin{conj}[Tate-Voloch \cite{TaVo}]
Let $A$ be a semiabelian variety over $\C_p$ and $X$ a closed subscheme. Then there is an $\e>0$ such that for any torsion point $P\in A(\C_p)_\tor$, either $P\in X(\C_p)$ or $d(P,X)>\e$.
\end{conj}

In \cite{TaVo} Tate and Voloch also proved their conjecture in the situation where the ambient variety is a torus. The general case of the conjecture is
still open. However, the conjecture is proven when the semiabelian variety $A$ is defined over $\ov {\Q_p}$. This is due to T. Scanlon, who first proved the result for points of order prime to $p$ (see \cite{Scan98}) and then for the entire torsion subgroup, using
a result of Sen in p-adic Hodge theory (see \cite{Scan99}). Both proofs use algebraic model theory along the same lines as Hrushovski in his proof of the Manin-Mumford conjecture (see \cite{Hru} for the latter). R\"ossler then gave some effectiveness to a restricted form of this result in \cite{RosTV}, using Arakelov geometry.

The analogous problem for Drinfeld modules was studied and solved by Ghioca (see \cite{GhioTaVo}), and more recently, in the context of the Andr\'e-Oort conjecture, Habegger proved an analogous result for CM points with ordinary reduction in a power of the modular curve $Y(1)$ (see \cite{Habegger}), thus suggesting that there might exist a "modular" Tate-Voloch conjecture for special points in mixed Shimura varieties. However, such a conjecture must be stated carefully since the naive formulation would imply a ``Mordell-Lang-Tate-Voloch'' conjecture, which is false (see Remark \ref{Habeg} for a counter-example).

\bigskip

In this paper, we will give a new proof of (a weaker form of) the result of Scanlon:

\begin{thm}
Let $K$ be a finite extension of $\Q_p$ and let $A$ be a semiabelian variety over
$K$. Let $X\inj A$ be a closed subscheme defined over $K$. Then there is an $\e>0$ such that for any torsion point $P\in A(\C_p)_\tor$, either $P\in X(\C_p)$ or $d(P,X)>\e$.
\label{mainthm}
\end{thm}

\begin{rem}
 This is weaker than Scanlon's result, since we assume that the subvariety $X$ is defined over $K$ and not merely over $\C_p$. The proof we give makes heavy use of this fact, but we are confident in the fact that it could be altered to include the more general case.
\end{rem}

Our proof is based on a study of the Galois action on torsion points and so-called torsion subschemes and follows the pattern of the algebro-geometric proof of the the Mordell-Lang conjecture given in \cite{MM2ML}. Instead of appealing to
Sen's result (see above), we use the structure of Galois groups of local fields, together with a combinatorial device known as Boxall's lemma.

In the first section, we define the distance functions that we will use.
In section two, we deal with the unramified torsion and in section three we deal with ramified prime-to-$p$ torsion. Finally in section four, we tackle the issue of ramified $p$-primary torsion and
pulling all our strings together we obtain a proof of Theorem \ref{mainthm}.


\begin{rem}\label{Habeg}
 Looking at the relationship between the Manin-Mumford and Mordell-Lang conjectures, we might wonder if there is a ``Tate-Voloch conjecture for finitely generated subgroup''. However the following example of Scanlon's given to us by Habegger shows that it is not true that a sequence of points in a finitely generated subgroup cannot approach a subvariety: for $p$ an odd prime, the subgroup generated by $2\in\mb G_m(\Q_p)$ has $1$ as one of its limit points because for all $n>0$,
 \[\lvert 2^{(p-1)p^{n-1}}-1\rvert_p\leq p^{-n}.\]
 Hence, if an analogy could be made in a ``Mordell-Lang'' setting (or more generally in a ``Pink-Zilber'' setting), it would need a more refined statement.
\end{rem}

\section{The {$p$}-adic distance on schemes}
Let $K$ be a complete subfield of $\mb C_p$, and $A$ be a scheme of finite type over $K$. Let $\mc O_K$ be the valuation ring of $K$.

\subsection{Defining {$p$}-adic metrics}
For any closed $K$-subscheme $X\subset A$, we may define a distance from a point $P\in A(\mb C_p)$ to $X$ as following : if $A=\mb A^n_K$, and $I\subset K[x_1,\dots,x_n]$ is the ideal defining $X$, we put
\[d_{\mb A^n}(P,X)=\max\left\{\lvert f(P)\rvert_p,f\in I\cap \mc O_K[x_1,\dots,x_n]\right\}\]
 where $\lvert\cdot\rvert_p$ denotes the usual $p$-adic norm on $K$; in general, if $A=\bigcup_{i\leq r}U_i$ is a decomposition of $A$ in open affine subschemes, we define
 \[d^{\mc U}_A(P,X)=\max \left\{d_{U_i}(P,X\cap U_i),P\in U_i\right\}.\]
We will omit the subscript $A$ if it is clear from context.

For a point $Q\in A(\C_p)$, we can consider it as a reduced closed subscheme of $A$, thus defining $d^{\mc U}(P,Q)$, which is a distance in the usual sense, and gives a metric to the topological space $A(\C_p)$.

This construction obviously depends on the decomposition $\mc U$; however, for two such decompositions $\mc U$ and $\mc V$, the distances are equivalent, in the sense that we have positive real numbers $0<\al<\beta$ such that
\begin{equation}\label{equi}\tag{E}
\al d^{\mc U}(P,X)\leq d^{\mc V}(P,X)\leq \beta d^{\mc U}(P,X).
\end{equation}
for all $P\in A(K)$. This implies that for a sequence $(P_n\in A(K))_{n\in\N}$,
the statements $\lim_{n\to\infty} d^{\mc U}(P_n,X)=0$ and $\lim_{n\to\infty} d^{\mc V}(P_n,X)=0$ are equivalent.

\bigskip

When $A$ has a model $\mc A$ over $\mc O_K$, the ring of integers of $K$, we have another description for this: let $S=\Sp \mc O_K$ and for $0<\e\leq 1$, $S_{(\e)}$ be the closed subscheme defined by the ideal $I_\e:=\{a\in \mc O_K, \lvert a\rvert_p<\e\}$. Assume we are given an $S$-scheme $\mc A$, and a closed subscheme $\mc X$. For any point $P\in \mc A(S)$ and any $0<\e\leq 1$, we have a pull-back $P_\e\in\mc A(S_{(\e)})$. We define
\[d_{\mc A}(P,\mc X)=\inf\{\e\in(0,1],P_\e\in \mc X(S_{(\e)})\}.\]

Since $K$ is complete, $d_{\mc A}(P,\mc X)=0$ precisely for points $P\in\mc X(S)$.

For simplicity, and when there is no ambiguity, we will omit the subscript $\e$ on points, so for $P\in\mc A(S)$, the notation $P\in \mc X(S_{(\e)})$ means ``$P$ reduces in $\mc A(S_{(\e)})$ to an element of $\mc X(S_{(\e)})$''.

For ease of computation, we will define $d_{\mc A}(P,\mc X)=1$ for $P\in A(K)\setminus \mc A(S)$. This is coherent with taking distances in a proper completion of $\mc A\to S$ (when such a completion exists), where $A(K)=\mc A(S)$ by the valuative criterion for properness.

 One can check that if $\mc U$ is an affine covering of $A$, coming from an $S$-affine covering of $\mc A$, then $d_{\mc A}(P,\mc X)=d^{\mc U}_A(P,X)$ where $X$ is the generic fibre of $\mc X$. This implies that distances coming from two different models are equivalent (in the sense of (\ref{equi})).

\bigskip

A good reference for all these notions of distance is \cite{VolDist}. We recall the following useful properties.

\begin{prop}\label{TV}
\begin{enumerate}[i)]
\item For subschemes $X$ and $Y$ closed in $A$, and any point $P\in A(\C_p)$, $d(P,X\cap Y)=\max(d(P,X),d(P,Y))$.
\item If $f:\mc A\to\mc B$ is a morphism of $S$-schemes then for any closed subscheme $\mc Y\inj\mc B$ and $P\in\mc A(S)$, and letting $\mc X=f^*\mc Y$, we have
\[d_{\mc A}(P,\mc X)=d_{\mc B}(f(P),\mc Y).\]
\item If $L$ is a Galois extension of $K$ with ring of integers $\mc O_L$ and $\mc X$ is a subscheme of $\mc A':=\mc A\tens_{\mc O_K}\mc O_L$, we have
\[\forall\s\in\Gal(L/K),\ d_{\mc A'}(\s(P),\mc X^\s) = d_{\mc A'}(P,\mc X) = d_{\mc A}(P,\mc X).\]

\end{enumerate}

\end{prop}
\begin{proof}
 Exercise to the reader. See properties (b), (c) and (e) from \cite[Theorem 1]{VolDist}.
\end{proof}

From now on, we will assume that $K$ is a finite (hence complete) extension of $\Q_p$ so $\mc O_K$'s residue field $k$ is finite, say of cardinality $q$, and we can define
\[\e_K:=p^{-1/e}=\max\{\lvert x\rvert_p,\ x\in\mf m\}<1,\]
where $e$ is the ramification index of $K/\Q_p$ and $\mf m$ is the maximal ideal in $\mc O_K$. So for any data $\mc X\inj \mc A$ above $S$, $P\in \mc A(S)$, we have
\[d_{\mc A}(P,\mc X)\leq \e_K^{n}\ssi P\in\mc X(S_n).\]
Where we have used the notation $S_n=S_{(\e_K^n)}=\Sp \mc O_K/\mf m^{n+1}$.
Writing $S^\sh=\Sp\mc O^\sh_K$, the strict henselization of $\mc O_K$, we have the analog description for $P\in \mc A(S^\sh)$
\[d_{\mc A}(P,\mc X)\leq \e_K^{n}\ssi P\in\mc X(S^\sh_n),\]
 since $\mc O^\sh_K$ is the ring of integers of the maximal unramified extension $K^\nr$ of $K$, which has the same valuation group as $K$.

\subsection{Distance of torsion points in a semiabelian variety}

Now, keeping the notation from the previous subsection, let us assume that $A$ is a semiabelian variety over $K$. In this setting we will use the following notations, which we will keep until the end of this article: for an integer $n$, a closed subscheme $X$ and a point $P\in A(\ov K)$, $[n]$ is the multiplication-by-$n$ morphism, $n\cdot X$ is the direct image of $X$ by $[n]$ (which is closed because $[n]$ is proper), and $X^{+P}$ is the translated subscheme. Following \cite[VIII, \S6]{SGA3b}, we also define the stabilizer $\Stab(X)$ of $X$ as the closed subgroup variety whose geometric points $P$ satisfy $X^{+P}=X$.


The theorem that we intend to prove is the following:

\begin{thm}
For any $X\inj A$ and affine covering $\mc U$ of $A$, there is an $\e>0$ such that for any $P\in A(\C_p)_\tor$ either $P\in X(\C_p)$ or $d^{\mc U}_A(P,X)>\e$.
\end{thm}

The structure of our proof is as follow : first we prove the case where $X$ is reduced to a point, which is known as Mattuck's theorem. Then we prove that we can reduce to the case where $X$ is an irreducible subscheme of $A$ with trivial stabilizer. Then we separate the cases of unramified torsion points, torsion points with only prime-to-$p$ ramification, and torsion points with $p$-primary ramification to prove the general case by induction on the dimension of $X$.

\bigskip

Notice that, as we said before, the choice of an affine covering only changes the value of $\e$; we may therefore work with the covering that suits us best, and in particuliar use $d_{\mc A}$ for any model $\mc A$ of $A$. In that case, we will write $A_0$ for the special fibre of $\mc A$.

\begin{rem}
From the Proposition \ref{TV}.ii), it is clear that for $P,Q\in \mc A(S)$,
\[d_{\mc A}(P-Q,\mc X) = d_{\mc A}(P,\mc X^{+Q}).\]
\end{rem}

The main ingredient in the proof of Mattuck's theorem is the following:

\begin{prop}\label{formal}
Let $\mc A\to S$ be a semiabelian scheme.
 For any algebraic extension $K'/K$, with ring of integers $\mc O_{K'}$ and residue field $\ka$,

a) The kernel of the reduction map $\pi:\mc A(\mc O_{K'})_\tor\to A_0(\ka)$ is discrete for the $p$-adic topology;

b) If $K'/K$ is unramified, $\ker \pi$ is a finite group.
\end{prop}
\begin{proof}
 a) Let $L$ be the completion of $K'$ with ring of integers $\mc O_L$. The kernel of the full reduction map $\mc A(\mc O_L)\to A_0(\ka)$ is isomorphic to the group $\wh A(\mf m_L)$ of points of the formal group $\wh A$ of $A$ with coordinates in the maximal ideal $\mf m_L$ of $\mc O_L$. However, the formal logarithm gives an isomorphism from an open subgroup $G\subset\wh A(\mf m_L)$ to $\wh{\mb G_a}^d(\mf m_L)$ where $\wh{\mb G_a}$ is the additive formal group. Since $\wh{\mb G_a}$ has no torsion, $\ker \pi$ must be discrete (a closer analysis would show that it is a $p$-group, but we will not need it here).

b) If $K'/K$ is unramified, the open subgroup $G\subset\mf A(\mf m_L)$ must be of finite index (see \cite[III, \S 7]{BourbLie23}), so $\ker \pi$ is finite.
\end{proof}

\begin{lem}[Mattuck \cite{Mattuck}]\label{zero}
The conjecture is true when $X$ is the $0$ point of $A$ seen as a reduced closed subscheme. Equivalently, the topological group $A(\ov K)_\tor$ is discrete.
\end{lem}
\begin{proof}
Clearly, the statement of the conjecture for $X=0$ is that it is an isolated point in $A(\ov K)_\tor$, which means that this group is discrete for the $p$-adic topology.

Let $\mc A$ be a semiabelian model for $A$ over $S$. This exists after a finite extension of the base field $K$ according to the semi-stable reduction theorem for semiabelian varieties \cite[Th\'eor\`eme 3.6]{SGA7a}. The set $\{P\in A(\ov K)_\tor,\ d_{\mc A}(P,0)<1\}$ is precisely the kernel of the homomorphism $\pi$ defined in the previous proposition, so a neighborhood of $0$ is discrete in $A(\ov K)_\tor$, hence the whole group must also be discrete.
\end{proof}


From now on, we will assume that $X$ is an irreducible subscheme of $A$ with trivial stabilizer. This can be done without any loss of generality, because if $P\in A(\ov K)_\tor$ is close to a $X$, then it must be close to one of its irreducible components, and furthermore, by Proposition \ref{TV}.$ii)$, if $\Stab(X)=G\neq 0$, we may work in $X/G\inj A/G$ since
\[d_{A}(P,X)=d_{A/G}(\wt P,X/G),\]
where $\wt P$ is the reduction of $P$ in $A/G$.

\bigskip

\begin{rem}\label{ucosets}
Consequently to this reduction and Mattuck's theorem, the conjecture still holds if $X$ is a finite union of cosets of subgroupschemes, \emph{ie} subschemes of the form
\[\bigcup_{i\in I} G_i^{+a_i},\]
where $I$ is a finite index set, the $a_i$ are geometric points and the $G_i$ are group subschemes. Those subschemes, which will be of an interest to us, are called torsion subschemes, in reference to the Manin-Mumford conjecture.
\end{rem}

We will also make heavy use of the following lemma:

\begin{lem}\label{multiple}
 Suppose $X$ is irreducible and has trivial stabilizer in $A$ and let $N$ be a positive integer. If $Q_n\in A(\ov K)\setminus X(\ov K)$ is a sequence converging to $X$, then $[N]Q_n$ converges to $N\cdot X$ and either a subsequence of $Q_n$ converges to a proper closed subset of $X$ or there is a subsequence such that $d([N]Q_n,N\cdot X)>0$.
\end{lem}

\begin{proof}
 It is clear from Proposition \ref{TV} that $d([N]Q_n,N\cdot X)$ converges to $0$. Now, suppose that $[N]Q_n\in N\cdot X(\ov K)$ for almost all $n$. Then $Q_n \in X(\ov K) + \ker[N]$, and since this kernel is finite, we may find $h\in\ker[N]$ and a subsequence such that $Q_n\in X^{+h}(\ov K)$ so $d(Q_n,X\cap X^{+h})$ converges to $0$. Since $\Stab(X)=0$, $X\cap X^{+h}$ is a proper closed subset of $X$ so the lemma is proved.
\end{proof}

\section{The unramified torsion}

As seen in the previous section, we consider a closed subvariety $X$ of a semiabelian variety $A$ both defined over a $p$-adic field $K$.

Until the end of this section, we restrict our problem to unramified torsion points, \emph{ie} points in $A(K^\nr)_\tor$. However, since it changes nothing to the truth of the conjecture, we will allow a change of base field to a (possibly ramified) finite extension of $K$ if necessary. Once $K$ is fixed, we choose $\tau\in\Gal(K^\nr/K)$ to be a lift of the Frobenius automorphism in $\Gal(\ov k/k)$ (recall that the residue field $k$ is finite).

\subsection{A Galois equation}

\begin{prop}\label{galois}
 For any semiabelian variety $A$ over $K$ there is a finite extension $K'$ of $K$ such that we may find a positive integer $N$ and a monic polynomial $F\in\Z[T]$ with no cyclotomic factors, such that
\[\forall P\in A(K'^\nr)_\tor,\ [N]F(\tau)(P)=0.\]
\end{prop}
\begin{rem}
 Saying that $F$ has no cyclotomic factors is equivalent to saying that it has no complex root which is a root of unity.
 \end{rem}

\begin{proof}
It is quite straightforward that this property is stable under group extensions, meaning that if we have a short exact sequence of group varieties
\[0\to G' \to G \to G''\to 0,\]
the property holds for $G$ if and only if it holds for $G'$ and $G''$ (if $F'$ and $F''$ are polynomials for $G'$ and $G''$ respectively, $F'F''$ is a polynomial for $G$, and the same holds for integers $N',N''$ and $N$).

Furthermore, after a finite extension of $K$, we may assume that $A$ is an extension of an abelian variety with semiabelian Néron model by a split torus. Hence we only need to prove it for the multiplicative group variety $\mb G_{m,K}$ and abelian varieties with semiabelian Néron model.

\bigskip

So let us assume first that $A=\mb G_{m,K}$. Then it has an obvious model $\mc A = \mb G_{m,S}$ over $S$, whose special fibre satisfy the equation \[(\Frob-[q])=0\in\End(A_0(\ov k)).\]

Furthermore, $A(K^\nr)_\tor\subset \mc A(S^\sh)$ because roots of unity are integers in $K^\nr$, so we may use Proposition \ref{formal}.b) to fix an integer $N$ such that $[N]\in\End(A(K^\nr)_\tor)$ factors through the reduction map $\pi:A(K^\nr)_\tor\to A_0(\ov k)$.

The last two results implies that $N(T-q)\in\Z[T]$ satisfies the desired equation (it is also obvious that it has no cyclotomic factors).

\bigskip

Suppose now that $A$ is an abelian variety, which has a semiabelian Néron model $\mc A\to S$. The Weil conjectures for semiabelian varieties over a finite field \cite{WeilConj} gives us the existence of a monic polynomial $F_0$ such that $F_0(\Frob)=0\in\End(A_0(\ov k))$. All of its complex roots are of absolute value $\sqrt{q}$ or $q$, so it cannot have cyclotomic factors.

Since $A(K^\nr)=\mc A(S^\sh)$, we may once again use \ref{formal},b) as above, to induce that $N\cdot F_0(T) \in\Z[T]$ satisfy the desired equation for the same integer $N$.
\end{proof}

\begin{rem}
 From the proof, we can see that, given that the order of $\ker \pi$ (from Proposition \ref{formal}) is a power of $p$, so is $N$ in this proposition.
\end{rem}

\subsection{Torsion subscheme in {$A^d$}}

Now let us assume we have made the necessary base field extension to apply Proposition \ref{galois}, so that we are given a positive integer $N$ and a monic polynomial $F$. We will now work in a bigger semiabelian variety, namely $A^d$ where $d$ is the degree of the polynomial $F$.

Let $M$ be the companion matrix to $F$, \emph{ie} if $F(x)=x^d-a_{d-1}x^{d-1}-\dots-a_0$,
\[M=\left[
\begin{array}{cccc}
0 & 1 & 0  & 0 \\
\vdots & 0 & \ddots & \vdots \\
0 & \cdots & \cdots &   1 \\
a_0 & a_1 & \cdots &  a_{d-1}
\end{array}\right].\]

We see $M$ as an isogeny of $A^d$, so we may build a subscheme
\[Z= Z_{X,F}:=\bigcap_{n\geq 0} M^n_*\left(\bigcap_{r\geq 0} M^{r,*}(X^d)\right)\subset X^d.\]
 Here the intersection are schematic (\emph{ie} fibre products of the immersions in $A^d$) and the upper and lower star denotes respectively the usual pull-back and schematic image by the morphism. $X^d$ is the cartesian product of the $d$ images of $X$ in $A^d$.

 Notice the following interesting properties of this subscheme

\begin{prop}\label{sigmaconv}
 \begin{enumerate}[i)]
  \item Set theoretically, $M(Z)=Z$. In particular, $Z$ is a torsion subscheme of $A^d$ \emph{ie} a finite union of cosets of closed subgroup schemes of $A$.
  \item Let $\s\in\Gal(\ov K/K)$ and $x_n\in A(\ov K)$ be a sequence which converges to $X^d$ and satisfies $F(\s)(x_n)=0$ for all  $n$. Then $x_{n,\s}:=(x_n,\s(x_n),\dots,\s^{d-1}(x_n))$ converges to $Z$.
 \end{enumerate}
\end{prop}
 \begin{proof}
 Let us prove the set-theoretic equality $M(Z)=Z$. Since $M$ is proper, the underlying set of $Z$ can be described as
\[Z=\bigcap_{n\geq 0}M^n\left(\bigcap_{r\geq 0}M^{r,-1}(X^d)\right).\]
However, it is quite obvious that
\[M\left(\bigcap_{r\geq 0}M^{r,-1}(X^d)\right)\subset \bigcap_{r\geq 0}M^{r,-1}(X^d)\]
so we have a descending chain of subsets
\[\bigcap_{r\geq 0}M^{r,-1}(X^d)\supset M\left(\bigcap_{r\geq 0}M^{r,-1}(X^d)\right)\supset \dots\supset M^n\left(\bigcap_{r\geq 0}M^{r,-1}(X^d)\right)\supset\dots\]
which must stop descending by noetherianity, so there is an $l\geq 0$ such that
\[M^l\left(\bigcap_{r\geq 0}M^{r,-1}(X^d)\right)=M^{l+1}\left(\bigcap_{r\geq 0}M^{r,-1}(X^d)\right),\]
so the equality of sets $M(Z)=Z$ must be true.
Using \cite[Prop 6.1]{PinkRo}, the stability of $Z$ by $M$ implies that torsion points are dense (at least after pulling back to $\ov K$), so they are torsion subschemes by the Manin-Mumford conjecture for semiabelian varieties.

 For the second assertion, the Galois equation $F(\s)(x_n)=0$ implies that
 \[\s(x_{n,\s}):=(\s(x_n),\s^2(x_n),\dots,\s^d(x_n))=M(x_{n,\s})\]
 so the Galois action $\s$ is given by the algebraic morphism $M$ for all $x_{n,\s}$. Since $X^d$ is Galois invariant, we can compute

 \begin{equation*}
 \begin{split}
d\left(x_{n,\s},M^{r,*}(X^d)\right) = &\ d\left(M^r(x_{n,\s}),X^d\right) \\
= &\ d\left(\s^r(x_{n,\s}),X^d\right) \\
= &\ d\left(x_{n,\s},(X^d)^{\s^{-r}}\right) \\
= &\ d\left(x_{n,\s},(X^d)\right)
 \end{split}
 \end{equation*}

 so $x_{n,\s}$ must converge to the intersection $\ds{\bigcap_{r\geq 0} M^{r,*}(X^d)}$.
 Similarly, we deduce from this by using the Galois and algebraic equivariance of the distance that $x_{n,\s}$ converges to $Z$.
 \end{proof}

\subsection{Ultrafields}

In order to work with ramified torsion points, we will need to see the sequence $(P_n)$ as a single point in some large field over $K$. Let us first describe this setting and what result can be obtained by the use of ultrafilters.

The ideas in this subsection comes from an argument by Hrushovski given in \cite[p. 156]{Bousc}.

Let $\mf U$ be a fixed non-principal ultrafilter on $\N$. We give here the first properties of $\mf U$ as a subset of the powerset of $\N$:
\begin{itemize}
 \item $\N\in\mf U$;
 \item if $U\subset U'$ and $U\in\mf U$, then $U'\in\mf U$;
 \item if $U,V\in\mf U$, $U\cap V\in\mf U$;
 \item for any subset $U\subset\N$, either $U$ or $\N\setminus U$ is in $\mf U$, but not both;
 \item any cofinite set is in $\mf U$ (or equivalently, using the previous property, $\mf U$ contains no finite set).
\end{itemize}

For more results in the theory of ultrafilters and ultraproducts, see \cite{BourbGenTop}. Considering the last property, we will say that a predicate $P(n)$ on $\N$ is true for $\mf U$-almost all $n$ if the set of indices $\{n:\ P(n)\}$ is in $\mf U$.

\bigskip

Let $I_{\mf U}$ the ideal of the ring $R$ of bounded sequences in $\ov K$ defined by
\[(a_i)_{i\in\N}\in I_{\mf U} \ssi \forall\; n\in\N,\ \{i:\lvert a_i\rvert_p\leq p^{-n}\}\in\mf U.\]
This ideal is maximal in $R$, and we define $D=D_{\mf U}$ to be the quotient field. It has an obvious diagonal injection $\ov K\inj D$ and the Galois group action extends naturally. Notice that $D$ is a quotient of the ring of bounded elements in the usual ultrapower $\ov K^\N /\mf U$.

If we have a sequence of points $P_n\in A(\ov K)$, it defines naturally a point $P^*\in A(D)$. By construction, we have
\begin{equation}\label{ultrapoint}
 d(P_n,X)\xrightarrow{n\geq 0}0 \ssi P^*\in X(D).
\end{equation}

\begin{rem}
 In a way, going to the ultrafield $D$ allows us to get asymptotic information on the sequence $P_n$ without specifying one of the equivalent metrics described previously, thanks to the equivalence (\ref{ultrapoint}). Thus we do not need to choose any affine covering or integral model, which might prove difficult in the non proper case and adds information which is irrelevant for the problem at hand.
\end{rem}

\begin{prop}\label{inter}
 Let $F\in\Z[T]$ be a monic polynomial of degree $d$, and $X$ be a closed irreducible subvariety of a semiabelian variety $A$ with trivial stabilizer. Let $L$ be an algebraic extension of $K$. Then there is an $N\in\N$ such that for any $\s\in Gal(L/K)$ and any sequence $(P_n)$ in $A(L)$ whose distance to $X$ converges to $0$ satisfying
\[\forall\; n\in\N,\ F(\s)(P_n)=0,\]
 either a subsequence of $(P_n)$ converges to a proper closed subset of $X$ or $P^*\in X(D)$ satisfies the equation $(\s^N-1)(P^*)=0$.
\end{prop}

\begin{proof}

This is an adaptation of the main idea behind the proof of the Manin-Mumford conjecture in characteristic $0$ as given by Pink and R\"ossler in \cite{PinkRoHru}.

We fix the morphism $\pi:A^d\to A$ to be the first projection. The second result of Proposition \ref{sigmaconv} implies that $P_{n,\s}$ converges to $Z$. Note that, since $F(\s)(P_n)=0$, $\pi M^k(P_{n,\s})=\s^k(P_n)$.

Since $M$ is an isogeny, for some $m>0$, $M^m$ stabilizes each of the irreducible components of $Z$. Let $T$ be one of these; we may assume that there is a $\s$ and a subsequence $Q_n$ of $P_{n,\s}$ that converges to it, otherwise we may take it out of $Z$ without loosing any of our assumptions.

Since $\pi(Q_n)$ converges to $\pi_*(T)$, we may assume that this scheme is dense in $X$, or we would have a subsequence converging to a closed proper subset of $X$, namely the Zariski closure of $\pi_*(T)$ (which may not be closed if $\pi$ is not proper).

This implies that $\pi(\Stab(T))\subset\Stab(X)=0$, so the projection $\pi$ induces a morphism $\pi_T:T/\Stab(T)\to A$.

Since $T$ is stabilized by $M^m$, it is also true of its stabilizer, so $M^m$ induces an isogeny $M'$ of $A/\Stab(T)$ stabilizing $T/\Stab(T)$, so in particular a birational automorphism of $T/\Stab(T)$. However, this variety is of general type \cite{Ueno}, so its group of birational automorphisms is finite \cite{Matsu}, and $M'^k=1$ for some $k$.

Since this is true for any irreducible component $T$, we may assume that $k$ satisfies this equation for any irreducible component $T$ of $Z$.

\bigskip

Since $P_{n,\s}$ converges to $Z$, $P_\s^*\in T(D)$ for some irreducible component $T$ of $Z$, so we have (writing $\ov {P_\s^*}$ for the image of $P_\s^*$ in $T/\Stab(T)$)
\begin{equation*}
 \begin{split}
\s^{km}(P^*)= &\ \pi M^{km}(P^*_\s) \\
            = &\ \pi_T M'^k(\ov{P^*_\s}) \\
            = &\ \pi_T (\ov{P^*_\s}) \\
            = &\ \pi (P^*_\s) \\
            = &\ P^*
 \end{split}
\end{equation*}
Notice that $N=km$ depends only of the construction from $X$ and $F$ so it is independent of $\mf U$, $\s$ and the sequence $P_n$.

\end{proof}

\subsection{The Tate-Voloch conjecture for unramified torsion points}

We are now ready to prove the following statement
\begin{thm}\label{unr}
Let $X\inj A$ be a closed subscheme of a semiabelian variety over $K$. Then for any distance $d$ as defined in the first section, there is $\e>0$ such that for any $P\in A(K^\nr)_\tor$, either $P\in X$ or $d(P,X)>\e$.
\end{thm}

\begin{proof}

We procede by induction on the dimension of $X$. If it is $0$, then this is basically Mattuck's theorem.
Otherwise, assume $\Stab(X)=0$ (by taking the quotient if necessary), and let us assume we may find a contradicting sequence $Q_n\in  A(\mc O_K^\sh)$ whose distance to $X$ is strictly decreasing to $0$.
We put $\mc Y=N\cdot \mc X$, $Y=N\cdot X$ and $P_n=[N]Q_n$.
By Lemma \ref{multiple} and using the induction hypothesis, we may also assume (by taking a subsequence if necessary) that $P_n\notin Y(K^\nr)$.

The main proposition of the previous subsection then implies that we may find an integer $m$ such that $(\tau^m-1)(P^*)=0$ where $P^*$ is the image of the sequence $P_n$ in the ultrafield $D_{\mf U}$ for some (any) non-principal ultrafilter $\mf U$ on $\N$.

Since $F(T)$ and $T^m-1$ are coprime in $\Q[T]$ (because $F$ has no cyclotomic factor) this implies that, choosing $m'\in\N$ such that $m'\Z[T]\subset F(T)\Z[T] + (T^m-1)\Z[T]$, we have $[m']P^*=0$, so $P^*$ is a torsion point, which means that the sequence $P_n$ converges $p$-adically to a geometric point (because torsion points are defined on $\ov K$), which contradicts Mattuck's theorem.

Thus the theorem is proved.

%
\end{proof}
%

\section{The ramified prime-to-{$p$} torsion}

In order to prove the conjecture in the ramified case, we have to refine the previous result. Recall that for any torsion group $G_\tor$ we have
\[G_\tor = G_\pp \oplus G_\ptor.\]
Thus we may decompose any torsion point in its $p$-primary part and its prime-to-$p$ part. In this section we intend to prove the conjecture for points $P\in A(K^\nr)_\tor + A(\ov K)_\pp$, \emph{ie} when we assume the $p$-primary part is still unramified, thus only the prime-to-$p$ part may be ramified.

\subsection{Galois results}

This case will reduce to the previous one thanks to the following result on ramification for prime-to-$p$ torsion points:

\begin{prop}\label{galeq}
 For any semiabelian variety $A$ over a $p$-adic field $K$, there exists a finite extension $K'$ of $K$ such that,
\begin{equation}\label{ramif}
\forall\; P\in A(K'^\nr)_\pp,\ \forall\; \s\in\Gal(\ov {K'}/ K'^\nr),\ (\s-1)^3(P)=0.
\end{equation}
\end{prop}
\begin{proof}
 Over a suitable $K'$, we may assume that $A$ is the extension of an abelian variety $A'$ with semi-stable reduction by a split torus $T$. In $A'$, we have the analog assertion \cite[Prop. 5.8]{Abbes}
\[\forall\; P'\in A'(K'^\nr)_\pp,\ \forall\; \s\in\Gal(\ov {K'}/ K'^\nr),\ (\s-1)^2(P')=0\]
and in the torus, the prime to $p$ torsion is unramified, (because it is true in $\mb G_m$).

Therefore, $(\s-1)^2(P)$ reduces to $0\in A'(\ov {K'})$, so it lies in $T(\ov {K'})_\pp$, hence it is unramified and $(\s-1)^3(P)=0$.
\end{proof}

This equation will be a crucial element in the proof of the conjecture, and it also allows reduction to tame ramification, thanks to the following corollary, which is a close adaptation of its analog on abelian varieties (see \cite[Lemma A.1]{BaRi}):

\begin{coro}
 The prime-to-$p$ torsion points are defined over a tamely ramified extension of $K'$.
\end{coro}
\begin{proof}
 It is known \cite{SerreCorpsloc} that the wild ramification group $G:=\Gal(\ov {K'}/K'^\ta)$ is a $p$-profinite group. For any finite group $M\subset A(\ov {K'})_\pp$, the action of $G$ on $M$ must factor through a finite $p$-group $G'$, say of order $q=p^k$. Let $I=(T-1)^3\Z[T]+(T^q-1)\Z[T]$ so that if $F\in I$, $F(\s)M=0$.

 Euclid's algorithm in $\Q[T]$ shows that
\begin{align*}
q(T-1)& =(T-1)^3\left(\frac{(q-1)T+(q+1)}{2}Q(T)+q\frac{(q-1)^2}4\right)\\
 &\quad -(T^q-1)\frac{(q-1)T+(q+1)}{2},
\end{align*}

 where $Q\in\Z(T]$ is the euclidian quotient of $T^q-1$ by $(T-1)^3$.
 Hence, if $p$ is odd, $q(T-1)\in I$, and if $p=2$, $4q(T-1)\in I$.
 Therefore $p^{k'}(\s-1)M=0$ for some power $k'$ (which can be taken to be $k$ if $p$ is odd and $k+2$ if $p=2$), but $p$ is invertible in $M$ so $\s$ acts trivially on $M$. since this is true for all $\s$, and $A(\ov K)_\pp$ is clearly covered by its finite subgroups, the prime-to-$p$ torsion points must be at most tamely ramified.
\end{proof}

\subsection{The Tate-Voloch conjecture in $A(K^\nr)_\tor+ A(\ov K)_\pp$}

The previous subsection has given us a Galois equation, so we will be able to apply Proposition \ref{inter}. However, we will need some ``descent'' result to reduce to the unramified case. This is where the following lemma is used:

\begin{lem}\label{descent}
 Let $L'/L$ be a Galois extension of algebraic overfields of $K$ with topologically finitely generated Galois group $G$. Let $Q_n$ be a sequence of torsion points in $A(L')_\tor$.
 If $Q^*$ is fixed by $G$, then $Q_n\in A(L)$ for $\mf U$-almost all $n$.
\end{lem}

\begin{proof}
 First notice that the $G$-fixed field of the ultrapower $F:=L'^\N/\mf U$ is precisely $F^G=L^\N/\mf U$ (\emph{ie} the ultrapower on $L$, which is a natural subfield of $F$). Indeed, if $(x_n)\in F$ is fixed by $g\in G$, Then $g(x_n)=x_n$ for $\mf U$-almost all $n$. Since we may find a finite set of topological generators for $G$, and $\mf U$ being stable by finite intersection, $x_n\in L$ for $\mf U$-almost all $n$, so $(x_n)\in L^\N/\mf U$.

 Now, let us come back to $(Q_n)=Q^*\in A(D)$. The fact that it is $g$-invariant for some $g\in G$ means that for all $k$, $d(Q_n,g (Q_n))\leq p^{-k}$ for $\mf U$-almost all $n$. But the set of torsion points in $A$ is discrete for the metric by Mattuck's theorem, so $Q_n=g(Q_n)$ for $\mf U$-almost all $n$, so the sequence $(Q_n)$ defines a point in $A(F^G)$. Hence $Q_n\in A(L)$ for $\mf U$-almost all $n$ by the previous discussion.
\end{proof}

%

We are now able to state and prove the conjecture in this setting:

\begin{thm}\label{pp}
 Let $X\inj A$ be a closed subscheme of a semiabelian variety over $K$. Then for any distance $d$ as defined in the first section, there is $\e>0$ such that for any $P\in A(K^\nr)_\tor+ A(\ov K)_\pp$, either $P\in X$ or $d(P,X)>\e$.
\end{thm}

\begin{proof}
 As in the unramified case, we procede by induction on the dimension of $X$. Let $P_n$ be an allegedly contradicting sequence, so that $d(P_n,X)$ goes to $0$ and the set $\{n:\ P_n\notin X(\ov K)\}$ is not finite.

 Since a finite extension of $K$ does not change the truth of the result, we will assume that the $P_n$ satisfy Proposition \ref{galeq} and its corollary.

Let $m$ be a positive integer (that will be fixed later on); According to Lemma \ref{multiple} and by the induction hypothesis, the set $\{n:\ [m]P_n\notin mX(\ov K)\}$ is not finite, so we may choose a non-principal ultrafilter $\mf U$ such that $\{n:\ [m]P_n\notin m\cdot X(\ov K)\}\in \mf U$. We will use all the constructions with respect to this ultrafilter.


Since $P_n\in A(K^\ta)$, using Proposition \ref{inter} and the induction hypothesis implies that there is an $M$ (which does not depend on $\mf U$) such that $(\s^M-1)(P^*)=0$ for all $\s\in\Gal(K^\ta/K^\nr)$.

Let $m\in\N$ be such that $m(T-1)$ as an ideal in $\Z[T]$ is contained in the ideal
\[(T^M-1)\cdot\Z[T]+(T-1)^3\cdot\Z[T].\]
This is possible because $\gcd\left((T-1)^3,T^M-1\right)=T-1$ in $\Q[T]$. Notice that this $m$ is independent of the choice of $\mf U$ we have made (since $M$ is), so the definition of $\mf U$ is not self-referential.

Then for any $\s\in \Gal(\ov K/K^\nr)$, $(\s^M-1)(P^*)=0$, and $(\s-1)^3(P^*)=0$, which implies that $[m](\s-1)(P^*)=0$.


It is known \cite{SerreCorpsloc} that $\Gal(K^\ta/K^\nr)$ is isomorphic to a quotient of $\wh \Z$ by $\Z_p$, thus is topologically finitely generated. So $[m]P_n$ is unramified for $\mf U$-almost all $n$ by Lemma \ref{descent} with $L'=K^\ta$ and $L=K^\nr$, and Theorem \ref{unr} then implies that $\{n:\ [m]P_n\notin m\cdot X(\ov K)\}$ is finite, contradicting the fact that $\mf U$ is non-principal and the theorem is proved.
\end{proof}

\section{The general case}

\subsection{Ramified $p$-primary torsion}

In this subsection, we deal with sequences of torsion points who all have a ramified $p$-primary part. In this setting, the following simple combinatorial result found by Boxall allows us to avoid Sen's result from $p$-adic Hodge theory \cite{Sen}:

\begin{thm}[Boxall]
 Let $p$ be a prime number, $L$ a field of characteristic $0$, and $A$ a semiabelian variety over $L$, such that $A[p]\subset A(L)$.
 If $p=2$, we assume $A[4]\subset A(L)$.

Then, for any $Q\in A(\ov L)_{\ptor}\setminus A(L)$,
\[\exists\;\s\in\Gal(\ov L/L),x\in A[p]\setminus\{0\},\ (\s-1)(Q) = x.\]

\end{thm}

 \begin{proof}
 The following is an obvious adaptation of a very nice proof by Oesterl\'e detailed in \cite{RosNot} for abelian varieties and $p>2$.

 Let $Q\in A(\ov L)_{\ptor}\setminus A(L)$, $n\geq 1$ the smallest integer such that $[p^n]Q\in A(L)$, and $\s_1\in\Gal(\ov L/L)$ such that $\s_1([p^{n-1}]Q)\neq [p^{n-1}]Q$. For $1\leq i \leq n$, we put

 \begin{itemize}
 \item $Q_i = [p^{n-i}] Q$,
 \item $\s_i = \s_1^{p^{i-1}}$,
 \item $x_i = (\s_i-1)(Q_i)$.
 \end{itemize}

 Since $[p]x_1 = \s_1([p^n]Q)-[p^n]Q = 0$, $x_1\in A[p]\setminus\{0\}$ ($x_1\neq 0$ by definition of $\s_1$). We claim that for all $i$, $x_i = x_1$, so $x_n = \s_n(Q) - Q  \in A[p]\setminus\{0\}$ and the result is proved.

 Let us prove our claim by induction, so suppose $x_i = x_1$ for some $i<n$. Since $A[p]\subset A(L)$, for any $P\in A(\ov L)$, $[p]P = 0\Rightarrow (\s_i-1)(P) = 0$. Since
 \begin{equation}\label{eq1}
 [p^2](\s_i-1)Q_{i+1} = [p]x_i = 0,
\end{equation}
 we have
\begin{equation}\label{eq2}
  [p](\s_i-1)^2Q_{i+1} = 0
\end{equation}
which in turn implies
\begin{equation}\label{eq3}
  (\s_i-1)^3(Q_{i+1}) = 0.
\end{equation}

 If $p=2$, (\ref{eq1}) gives that $(\s_i-1)Q_{i+1}\in A[4]\subset A(L)$ so
 \[x_{i+1} = (\s_i^2-1)(Q_{i+1}) = (\s_i+1)(\s_i-1)(Q_{i+1}) = [2](\s_i-1)(Q_{i+1}) = x_i.\]

 If $p$ is odd, notice that
 \[T^p = 1 + p(T-1) + \frac{p(p-1)}{2}(T-1)^2 +(T-1)^3R(T)\]
for some $R(T)\in\Z[T]$ so
\[T^p-1\equiv p(T-1)\ \mod\ \left(p(T-1)^2,(T-1)^3\right).\]
Therefore, using (\ref{eq2}) and (\ref{eq3}), $x_{i+1} = (\s_i^p-1)(Q_{i+1}) = [p](\s_i-1)(Q_{i+1}) = x_i$, so the claim, and thus the theorem, are proved.
 \end{proof}

We use Boxall's theorem to prove:

\begin{thm}\label{ptor}
 Let $X\inj A$ be a closed subscheme of a semiabelian variety over $K$ such that $A[p]\subset A(K^\nr)$ (this is always possible after a finite extension of $K$). Let $P_n$ be a sequence of torsion points with the following assumptions:
\begin{itemize}
 \item $d(P_n,X)$ converges to $0$
 \item for all $n$, if $P_n=x_n+y_n$ is its decomposition in $A(\ov K)_\ptor \oplus A(\ov K)_\pp$, $x_n\notin A(K^\nr)$.
\end{itemize}
Then $P_n\in X(\ov K)$ for almost all $n$
\end{thm}

\begin{proof}
 As usual we prove this by induction on the dimension of $X$, which we assume to have trivial stabilizer. Suppose that $(P_n)$ does not satisfy the conclusion of the theorem, so there is a subsequence (which we take to be $(P_n)$ itself) which stays outside $X(\ov K)$. Let $x_n$ and $y_n$ be as above, and $k_n$ be the order of $y_n$ (which is prime to $p$). Also remember that the relation (\ref{ramif}) holds for all $y_n$.

\bigskip

Using Boxall's theorem, we may find $\s_n\in\Gal(\ov K/K^\nr)$ and $u_n\in A[p]\setminus\{0\}$ such that $(\s_n-1)(x_n)=u_n$. Since $A[p]$ is finite, we may find $u\in A[p]\setminus\{0\}$ such that $u=[k_n']u_n$ for infinitely many $n$, where $k'_n=k_n$ if $p=2$, and $k'_n=2k_n$ otherwise. We assume that this equality is true for all $n$, by taking a subsequence if necessary.

Moreover, using the relation (\ref{ramif}), and the fact that in $\Z[T]$
\[T^m-1\equiv m(T-1)+\frac{m(m-1)}2(T-1)^2\ \mod (T-1)^3,\]
we have
\[(\s^{k'_n}_n-1)y_n=(\s_n-1)[k'_n]y_n + (\s_n-1)^2\left[\frac{k'_n(k'_n-1)}2\right]y_n = 0.\]
The last equality holds because $[k'_n/2]y_n=[k_n]y_n=0$ if $p>2$, and if $p=2$, $k_n$ is odd so
\[\left[\frac{k'_n(k'_n-1)}2\right]y_n=\left[\frac {k_n-1} 2 k_n\right]y_n=0.\]

Therefore, we have

\begin{equation*}
\begin{split}
(\s^{k'_n}_n-1)P_n = & (\s^{k'_n}_n-1)x_n + (\s^{k'_n}_n-1)y_n \\
 = & [k'_n]u_n \\
 = & u.
\end{split}
\end{equation*}

Since $X^\s=X$ for any $\s$ in the ramification group, this implies that
\[d(P_n,X)=d(\s^{k'_n}_n(P_n),X)=d(P_n+u,X)=d(P_n,X^{-u}).\]
So $P_n$ converges to $X$ and also to $X^{-u}$. Since $X$ has trivial stabilizer, $X^{-u}$ is distinct from $X$ and the sequence converges to the intersection $X\cap X^{-u}$ which has strictly smaller dimension, which is impossible by the induction hypothesis.
\end{proof}

\subsection{The Tate-Voloch conjecture in $A(\ov K)_\tor$}

We are now able to give the final statement, which holds in the most general case.

\begin{thm}
 Let $X\inj A$ be a closed immersion in a semiabelian variety, defined over a $p$-adic field $K$. Then for any distance $d$ as defined in the first section, there is $\e>0$ such that for any $P\in A(\ov K)_\tor$, either $P\in X$ or $d(P,X)>\e$.
\end{thm}

\begin{proof}
 Assume $K$ is such that $A[p]\subset A(K^\nr)$, and such that the relation (\ref{ramif}) holds (this is possible by taking a finite extension of $K$). Then everything has already been proved, since, given a contradicting sequence, we may find a subsequence such that either Theorem \ref{pp} or Theorem \ref{ptor} applies and induces a contradiction.
\end{proof}

\begin{rem}
As we have said before, the use of ultrafields allows us to forget a specific metric to look only at the convergence of the sequence. However, most of the proof given in this paper should go through when applied to an integral model of the semiabelian variety (provided that we have one that is nice enough), if we look more closely at integral sections.
\end{rem}
\def\cprime{$'$}

\end{document}